\documentclass{amsart}

\usepackage{graphicx}
\usepackage{amsmath}%
\usepackage{amsfonts}%
\usepackage{amssymb}%
\usepackage{graphicx}
\usepackage{enumerate}
\usepackage{enumitem}

\usepackage{tikz}
\usepackage{pgfplots}
\pgfplotsset{compat=1.9}

\newtheorem{theorem}{Theorem}[section]

\theoremstyle{definition}
\newtheorem{definition}[theorem]{Definition}
\newtheorem{example}[theorem]{Example}

\newtheorem{corollary}[theorem]{Corollary}
\newtheorem{proposition}[theorem]{Proposition}
\newtheorem{question}[theorem]{Question}
\newtheorem{conjecture}[theorem]{Conjecture}

\theoremstyle{remark}
\newtheorem{remark}[theorem]{Remark}

\numberwithin{equation}{section}

\begin{document}

\title{Coincidence Point Sets in Digital Topology}

\author{Muhammad Sirajo Abdullahi}

\curraddr[(MS Abdullahi, P. Kumam and J. Abubakar)]{\textit{KMUTTFixed Point Research Laboratory},  KMUTT-Fixed Point Theory and Applications Research Group, SCL 802 Fixed Point Laboratory, Department of Mathematics, Faculty of Science, King Mongkut's University of Technology Thonburi (KMUTT), 126 Pracha-Uthit Road, Bang Mod, Thrung Khru, Bangkok 10140, Thailand}
\email{abdullahi.sirajo@udusok.edu.ng [M.S. Abdullahi]}
\email{poom.kumam@mail.kmutt.ac.th [P. Kumam]}
\email{abubakar.jamilu@udusok.edu.ng [J. Abubakar]}

\author{Poom Kumam} 
\curraddr{Center of Excellence in \textit{Theoretical and Computational Science (TaCS-CoE)},  Science Laboratory Building, King Mongkut's University of Technology Thonburi (KMUTT), 126 Pracha-Uthit Road, Bang Mod, Thrung Khru, Bangkok 10140, Thailand}
\email{poom.kumam@mail.kmutt.ac.th [P. Kumam]}

\address[MS Abdullahi and J. Abubakar]{Department of Mathematics, Faculty of Science, Usmanu Danfodiyo University, Sokoto, Nigeria}
\email{abdullahi.sirajo@udusok.edu.ng [M.S. Abdullahi]}
\email{abubakar.jamilu@udusok.edu.ng [J. Abubakar]}



\author{Jamilu Abubakar} 


\thanks{The first and third authors were supported by the ``Petchra Pra Jom Klao Ph.D. Research Scholarship from King Mongkut's University of Technology Thonburi".}




\subjclass[2010] {Primary: 47H10, 54E35; Secondary: 68U10}

\keywords{Digital topology, Coincidence point set, Common fixed point set, Digital continuous maps, Fixed points, Retractions}

\begin{abstract}
In this article, we investigate some properties of the coincidence point set of digitally continuous maps. Following the Rosenfeld graphical model which seems more combinatorial than topological, we expect to achieve results that might not be analogous to the classical topological fixed point theory. We also introduce and study some topological invariants related to the coincidence and common fixed point sets for continuous maps on a digital image. Moreover, we study how these coincidence point sets are affected by rigidity and deformation retraction. Lastly, we present briefly a concept of divergence degree of a point in a digital image.
\end{abstract}

\maketitle

\section{Introduction} 		      

Topology is a branch of Mathematics that studies the relationship between spaces, especially equivalence between them under continuous mappings. It provides a lot of ease to many applications by reducing cost of computation through providing theoretical foundations and methods more efficient than the non topological ones. Fixed point theory in particular, plays an important and fundamental role in numerous areas of mathematics including functional and mathematical analysis, pure and applied topology, fuzzy theory etc. It has always provided us with a major theoretical tool in fields as widely as differential equations, topology, economics, game theory, dynamics, optimal control and functional analysis which leads to various and important applications in mathematics and applied sciences.

In metric spaces, this theory begins with the Banach fixed point theorem \cite{Banach} (also known as the Banach contraction principle), which guarantees the existence and uniqueness of a fixed point of a certain map $f: X \longrightarrow X$ of a complete metric space $X$, it additionally provides a constructive method of finding such a fixed point of the map $f$. For this direction (see. \cite{AbdullahiAzam17A, AbdullahiPoom18, Azametal09, Nadler, SintuKumam11}).

Topologically, the tools of fixed point theory are: the Lefschetz number, fixed point index, Nielsen number and the topological degree (for root problems). In classical topology, the value of $MF(f)$ (i.e. the minimum number of fixed points in the homotopy class of $f$) is generally hard to compute. The Lefschetz number $L(f)$ and the Nielsen number $N(f)$ (homotopy invariant lower bound for $MF(f)$) are often used to obtain $MF(f)$, where the former is homological in nature and gives a very rough indication of homotopy invariant fixed point information, while the later is more sophisticated and geometrical in nature. \cite{Jiang83lectures}.

On the other hand, digital topology deals with the questions of how and to what extent that topological concepts can meaningfully and usefully be applied to a binary image \cite{Kongetal92guest}. It is mainly concerned with studying mathematical properties of $n$-dimentional digital images \cite{Rosen79digital}. This study was initiated in the early 1970s by Azriel Rosenfeld \cite{Rosen70connectivity} (see also \cite{Rosen79digital}) and Mylopoulos and Pavlidis \cite{MyloPav71topological}. It has since provided the theoretical foundations for important image processing operations such as object counting, image thinning, image segmentation, boundary detection, contour filling, computer graphics and mathematical morphology etc (see. \cite{Bertrand94simple, Han05nonproduct, KongRosen96topological}).

A digital curve can be described as a sequence of digital points, or equivalently as a path of vertices on a graph \cite{Chen04discrete, Herman98geometry}. In general, we can define a digital surface based on direct adjacency and indirect adjacency \cite{Chen14digitalndiscrete}. The concept of digital surfaces was proposed by Artzy \textit{et. al.} \cite{Artzyetal81theory}, where they defined it as the face of some solid object. In 1981, Morgenthaler and Rosenfeld gave a different definition of digital surfaces \cite{MorgenRosen81surfaces}. They stated that a digital surface locally splits a neighborhood into two disconnected components. They also gave some classification results, which later Kong and Roscoe \cite{KongRoscoe85continuous} investigated further and concluded that most of those do not exist in terms of real world examples. This motivated Chen and Zhang \cite{Chen04discrete} to give another definition mainly for (6,26)-surfaces, called parallel-move based surfaces. They also obtained and proved the digital surface classification theorem \cite{Chen04discrete} (see also \cite{Chen14digitalndiscrete}). This inspired Chen and Rong \cite{ChenRong10digital} to calculate the genus and homology groups of 3-dimensional digital objects with the help of the classical Gauss-Bonnet Theorem and the Alexander Duality respectively.

An $n$-dimensional manifold is a topological space where each point has a neighborhood that is homeomorphic to an $n$-dimensional Euclidean space. In 1993, Chen and Zhang proposed a simple extension of digital surfaces to define a digital $n$-manifold \cite{ChenZhang93digital}. Melin \cite{Melin03connectedness} also studied digital $n$-manifolds using Khalimsky topological approach. Finding the orientability of digital manifolds is very significant in topology, as it is used to determine if a manifold contains a Mobius band. The digital Mobius band was first discovered by Lee and Rosenfeld \cite{KletteRosen04digital}. Afterwards, Chen \cite{Chen04discrete} designed an algorithm for determining whether a digital surface is orientable or not.

Until late 1980's, all works in digital topology were based on a graph-theoretic approach rather than topological, in which binary images are made into graphs by imposing adjacency relations on $\mathbb{Z}^n$. For $2$-dimensional binary images, the most frequently used adjacency relation is the (8,4) adjacency relation. The major problem of the graph-based approach to digital topology is that of determining what adjacency relations on $\mathbb{Z}^n$ might reasonably be used. One would normally want to use adjacency relations such that fundamental topological properties of $\mathbb{Z}^n$ have natural analogues for the graphs obtained from binary images. In \cite{Kongetal92concepts} Kong \textit{et al.} addressed this problem for $\mathbb{Z}^2$ and $\mathbb{Z}^3$. See \cite{Kongetal92guest} and references therein, for more details.

In \cite{BoxerSta19fixed}, the authors examined some properties of the fixed point set of a digitally continuous function. They believed that digital setting requires new methods that are not analogous to those of classical topological fixed point theory, and hence obtained results that often differ greatly from standard results in classical topology. They introduced some topological invariants related to fixed points for continuous self-maps on digital images, and study their properties. Their main contribution is the fixed point spectrum $F(X)$ of a digital image. i.e. the set of all numbers that can appear as the number of fixed points for some continuous self-map.

Motivated by the work of Boxer and Staecker in \cite{BoxerSta19fixed}, and the fact that coincidence theory has been greatly influenced by fixed point theory, in this manuscript we will investigate some properties of the coincidence point set of digitally continuous maps. We also introduce and study some topological invariants related to coincidence point sets and common fixed point sets for continuous maps on a digital image. Moreover, we study how these coincidence point sets are affected by rigidity and deformation retraction. Further, since the Rosenfeld graphical approach we intend to follow, seems more combinatorial than topological, we similarly expect to often achieve results that were not necessarily analogous to the classical topological coincidence point theory.

The organization of the paper is as follows: Section 1 houses an introduction to this research direction. In Section 2, we reviewed some basic and background material needed for this study. We introduce coincidence point spectrum and present some of its properties with some examples in Section 3. In Section 4, we introduce common fixed point spectrum, highlight some of its properties and present some illustrative examples. Section 5 studies how retractions interact with the coincidence and common fixed point spectra. In Section 6, we introduce and study the divergence degree obtained from the complement of the coincidence point set. Finally, in Section 7 we state our concluding remarks.

\section{Preliminaries}

Let $\mathbb{N}$ and $\mathbb{Z}$ denote the sets of natural numbers and integers respectively. Let us also denote by $\#X$ the number of elements (i.e. the cardinality) of a set $X.$ From now on, we denote by $id$ and $c$, the identity map (i.e. $id(x) = x$ for all $x \in X$) and the constant map (i.e. $c(x) = x_0$ for all $x \in X$ with $x_0 \in X$ fixed) respectively.

Traditionally, a digital image is a pair $(X, \kappa),$ where $X \subset \mathbb{Z}^n$ for some $n \in \mathbb{N}$ and $\kappa$ is an adjacency relation on $X$, which is symmetric and antireflexive. Therefore, we may view a digital image $(X, \kappa)$ as a graph for which $X$ is the vertex set and $\kappa$ determines the edge set. Usually, $X$ is finite and the adjacency relation reflects some type of ``closeness" of the adjacent points in $\mathbb{Z}^n$. When these usual conditions hold, one may consider the digital image as a model of a black and white real world digital image in which the black points (i.e. foreground) are represented by the members of $X$ and the white points (i.e. background) by the members of complement of $X$ (i.e. $\mathbb{Z}^n \backslash X)$ \cite{BoxerSta19fixed}.

We write $x \leftrightarrow_{\kappa} y$ to indicate that $x$ and $y$ are $\kappa$-adjacent or $x \leftrightarrow y$ whenever $\kappa$ is understood or it is unnecessary to mention. Further, we use the notation $x \Leftrightarrow_{\kappa} y,$ to indicate that $x$ and $y$ are $\kappa$-adjacent or are equal and use $x \Leftrightarrow y$ whenever $\kappa$ is understood.

In this paper, we will use the following type of adjacency. For $t \in \mathbb{N}$ with $1 \leq t \leq n$, any 2 (two) points $p = (p_1, p_2, \ldots , p_n)$ and $q = (q_1, q_2, \ldots , q_n)$ in $\mathbb{Z}^n$ (with $p \not= q$) are said to be $\kappa(t, n)$ or $\kappa$-adjacent if at most $t$ of their coordinates differs by $\pm 1,$ and all others coincide. Note that, the number of points adjacent to any element of $\mathbb{Z}^n$ which we represent by the $\kappa(t, n)$-adjacency relation on $\mathbb{Z}^n$ is determined by the number $t \in \mathbb{N}$ and can be obtain by the following formula, which appears in \cite{Han05nonproduct}:
\[\kappa := \kappa (t, n) = \sum_{i=1}^t 2^i C_i^n, \hspace*{0.5cm} \mbox{ where } C_i^n = \frac{n!}{(n-i)! \, i!}.\]

Following the graph theoretic approach of studying $n$-dimensional digital images, we will use the notions of $\kappa$-adjacency relations on $\mathbb{Z}^n$ and a digital $\kappa$-neighborhood as have been extensively used in the literature. More precisely, using the $\kappa$-adjacency relations as defined above, we say that a digital $\kappa$-neighborhood of a point $p$ in $\mathbb{Z}^n$ is the set defined and denoted as \cite{Rosen79digital}:
\[N_{\kappa}(p) := \{q \, \mid \, q \Leftrightarrow_{\kappa} p\}.\]
Also, the following notation is often use to denote a kind of neighborhood, the so called deleted digital $\kappa$-neighborhood of a point $p$ in $\mathbb{Z}^n$ \cite{KongRosen96topological}.
\[N^{\ast}_{\kappa}(p) := N_{\kappa}(p) \, \backslash \, \{p\}.\]

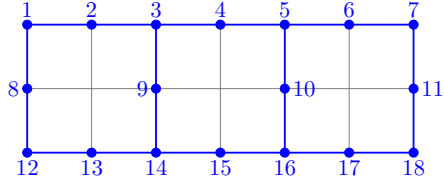
\begin{figure} \label{fig1}
\centering
\resizebox{6cm}{2.5cm}{%
\begin{tikzpicture}
	\draw[step=1cm,gray,very thin] (0,0) grid (6,2);
    \draw [blue, thick] (0,0) -- (0,2);
    \draw [blue, thick] (0,0) -- (6,0);
    \draw [blue, thick] (0,2) -- (6,2);
    \draw [blue, thick] (4,0) -- (4,2);
    \draw [blue, thick] (2,0) -- (2,2);
    \draw [blue, thick] (6,0) -- (6,2);
    \filldraw[blue] (0,0) circle (2pt) node[below]{12};
    \filldraw[blue] (0,1) circle (2pt) node[left]{8};
    \filldraw[blue] (0,2) circle (2pt) node[above]{1};
    \filldraw[blue] (1,0) circle (2pt) node[below]{13};
    \filldraw[blue] (2,0) circle (2pt) node[below]{14};
    \filldraw[blue] (3,0) circle (2pt) node[below]{15};
    \filldraw[blue] (4,0) circle (2pt) node[below]{16};
    \filldraw[blue] (5,0) circle (2pt) node[below]{17};
    \filldraw[blue] (6,0) circle (2pt) node[below]{18};
    \filldraw[blue] (1,2) circle (2pt) node[above]{2};
    \filldraw[blue] (2,2) circle (2pt) node[above]{3};
    \filldraw[blue] (3,2) circle (2pt) node[above]{4};
    \filldraw[blue] (4,2) circle (2pt) node[above]{5};
    \filldraw[blue] (5,2) circle (2pt) node[above]{6};
    \filldraw[blue] (6,2) circle (2pt) node[above]{7};
    \filldraw[blue] (2,1) circle (2pt) node[left]{9};
    \filldraw[blue] (4,1) circle (2pt) node[right]{10};
    \filldraw[blue] (6,1) circle (2pt) node[right]{11};
\end{tikzpicture}}
\caption{A 2-dimensional digital image with a 4-adjacency relation.}
\end{figure}

For $a, b \in \mathbb{Z}$ with $a \lneq b,$ the set $[a, b]_{\mathbb{Z}} = \{n \in \mathbb{Z} \, | \, a \leq n \leq b\}$ with $2$-adjacency relation is called a ``digital interval" \cite{Boxer94digitally}. We say that two subsets $(A, \kappa)$ and $(B, \kappa)$ of $(X, \kappa)$ are ``$\kappa$-adjacent" to each other if $A \cap B = \emptyset$ and there are points $a \in A$ and $b \in B$ such that $a$ and $b$ are $\kappa$-adjacent to each other. A set $X \subset \mathbb{Z}^n$ is called ``$\kappa$-connected" if it is not a union of two disjoint non-empty sets that are not $\kappa$-adjacent to each other \cite{Herman93oriented}. For a digital image $(X, \kappa)$ the ``$\kappa$-component" of $x \in X$ is defined to be the largest $\kappa$-connected subset of $(X, \kappa)$ containing the point $x$. 

\begin{definition} \cite{Rosen86continuous}
Let $(X, \kappa_1)$ and $(Y, \kappa_2)$ be digital images. A function $f : X \longrightarrow Y$ is $(\kappa_1, \kappa_2)$-continuous, if for every $\kappa_1$-connected subset $A$ of $X, f(A)$ is a $\kappa_2$-connected subset of $Y.$
\end{definition}

The function $f$ is called digitally continuous whenever $\kappa_1$ and $\kappa_2$ are understood. If $(X, \kappa_1) = (Y, \kappa_2)$ (i.e. $X=Y$ and $\kappa_1 = \kappa_2 = \kappa$) we say that a function is $\kappa$-continuous to abbreviate $(\kappa, \kappa)$-continuous.

\begin{theorem} \cite{Boxer99classical}
A function $f : X \longrightarrow Y$ between digital images $(X, \kappa_1)$ and $(Y, \kappa_2)$ is $(\kappa_1, \kappa_2)$-continuous if and only if for every $x, y \in X, f(x) \Leftrightarrow_{\kappa_2} f(y)$ whenever $x \leftrightarrow_{\kappa_1} y$.
\end{theorem}

\begin{theorem} \cite{Boxer99classical} \label{th2}
Let $f : X \longrightarrow Y$ and $g : Y \longrightarrow Z$ be continuous functions between digital images $(X, \kappa_1), (Y, \kappa_2)$ and $(Z, \kappa_3)$. Then $g \circ f : (X, \kappa_1) \longrightarrow (Z, \kappa_3)$ is continuous.
\end{theorem}

\begin{definition} \cite{Khalimsky87motion}
A digital $\kappa$-path in a digital image $(X, \kappa)$ is a $(2, \kappa)$-continuous function $\gamma : \lbrack 0, m \rbrack_{\mathbb{Z}} \longrightarrow X$. Further, $\gamma$ is called a digital $\kappa$-loop if $\gamma(0) = \gamma(m),$ and the point $p = \gamma(0)$ is the base point of the loop $\gamma.$ Moreover, $\gamma$ is called a trivial loop if $\gamma$ is a constant function.
\end{definition}

For a digital image $(X, \kappa),$ we define
\[C(X, \kappa) = \{f : X \longrightarrow X \, | \, f \mbox{ is $\kappa$-continuous}\}.\]

Recall that, a topological space $X$ has the fixed point property (FPP, for short) if every continuous function $f : X \longrightarrow X$ has a fixed point. A similar definition has appeared in digital topology as follows:
\begin{definition} \cite{Rosen86continuous}
A digital image $(X, \kappa)$ has the fixed point property (FPP) if every $\kappa$-continuous $f : X \longrightarrow X$ has a fixed point.
\end{definition}

\begin{figure}    
\centering
\begin{tikzpicture}
	\draw[step=1cm,gray,very thin] (0,0) grid (6,2);
    \draw [blue, thick] (0,0) -- (0,2);
    \draw [blue, thick] (0,0) -- (6,0);
    \draw [blue, thick] (0,2) -- (6,2);
    \draw [blue, thick] (4,0) -- (4,2);
    \draw [blue, thick] (2,0) -- (2,2);
    \draw [blue, thick] (6,0) -- (6,2);
    \draw [red, thick] (0,1) -- (1,2) -- (2,1) -- (3,2) -- (4,1) -- (5,2) -- (6,1) -- (5,0) -- (4,1) -- (3,0) -- (2,1) -- (1,0) -- (0,1);
    \filldraw[blue] (0,0) circle (2pt) node[]{};
    \filldraw[blue] (0,1) circle (2pt) node[]{};
    \filldraw[blue] (0,2) circle (2pt) node[]{};
    \filldraw[blue] (1,0) circle (2pt) node[]{};
    \filldraw[blue] (2,0) circle (2pt) node[]{};
    \filldraw[blue] (3,0) circle (2pt) node[]{};
    \filldraw[blue] (4,0) circle (2pt) node[]{};
    \filldraw[blue] (5,0) circle (2pt) node[]{};
    \filldraw[blue] (6,0) circle (2pt) node[]{};
    \filldraw[blue] (1,2) circle (2pt) node[]{};
    \filldraw[blue] (2,2) circle (2pt) node[]{};
    \filldraw[blue] (3,2) circle (2pt) node[]{};
    \filldraw[blue] (4,2) circle (2pt) node[]{};
    \filldraw[blue] (5,2) circle (2pt) node[]{};
    \filldraw[blue] (6,2) circle (2pt) node[]{};
    \filldraw[blue] (2,1) circle (2pt) node[]{};
    \filldraw[blue] (4,1) circle (2pt) node[]{};
    \filldraw[blue] (6,1) circle (2pt) node[]{};
\end{tikzpicture}
\caption{A 2-dimensional digital image with an 8-adjacency relation.}
\end{figure}

However, this property turns out to be very trivial, since the only digital image with the fixed point property (FPP) is a single point as was established in \cite{Boxeretal16digital} as follows:

\begin{theorem} 
A digital image $(X, \kappa)$ has the FPP if and only if $\#X = 1.$
\end{theorem}

\begin{definition} \cite{Boxer94digitally}
A function $f : X \longrightarrow Y$ between digital images $(X, \kappa_1)$ and $(Y, \kappa_2)$ is called an isomorphism if $f$ is a digitally continuous bijection such that $f^{-1}$ is digitally continuous.
\end{definition}

\begin{definition} \cite{Boxer99classical}
Let $(X, \kappa_1)$ and $(Y, \kappa_2)$ be digital images. Suppose that $f, g : X \longrightarrow Y$ are $(\kappa_1, \kappa_2)$-continuous functions, there is a positive integer $m$ and a function $H : X \times \lbrack 0, m \rbrack_{\mathbb{Z}} \longrightarrow Y$ such that:
\begin{enumerate} [label=\arabic*. ]
     \item For all $x \in X, H(x, 0) = f(x)$ and $H(x, m) = g(x)$;
     \item For all $x \in X,$ the induced function $H_x : \lbrack 0, m \rbrack_{\mathbb{Z}} \longrightarrow Y$ defined by 
     \[H_x(t) = H(x, t), \mbox{ for all } t \in \lbrack 0, m \rbrack_{\mathbb{Z}} \]
     is $(c_1, \kappa_1)$-continuous. That is, $H_x(t)$ is a $\kappa$-path in $Y$;
     \item For all $t \in \lbrack 0, m \rbrack_{\mathbb{Z}},$ the induced function $H_t : X \longrightarrow Y$ defined by
     \[H_t(x) = H(x, t), \mbox{ for all } x \in X\]
     is $(\kappa_1, \kappa_2)$-continuous. 
\end{enumerate}
Then $H$ is a digital homotopy (or $\kappa$-homotopy) between $f$ and $g$. Thus, the functions $f$ and $g$ are said to be digitally homotopic (or $\kappa$-homotopic) and denoted by $f \simeq g.$
\end{definition}
Note that if $m = 1,$ then $f$ and $g$ are said to be $\kappa$-homotopic in one step.

\begin{definition} \cite{Khalimsky87motion}
A continuous function $f : X \longrightarrow Y$ is called digitally nullhomotopic in $Y$ if $f$ is digitally homotopic to a constant function $c$. Moreover, a digital image $(X, \kappa)$ is said to be digitally contractible (or $\kappa$-contractible) if its identity map $id$ is digitally nullhomotopic.
\end{definition}

\begin{definition} \cite{BoxerSta19fixed, Haarmann15homotopy}
A function $f : X \longrightarrow Y$ is called rigid if no continuous map is homotopic to $f$ except $f$ itself. Moreover, when the identity map $id : X \longrightarrow X$ is rigid, we say that $X$ is rigid.
\end{definition}

\section{Coincidence Point Spectrum}

In \cite{BoxerSta18remarks}, the authors gave a brief treatment of homotopy-invariant fixed point theory. Following suit, we will now give a more general view of their treatment by extending it to a more general concept namely; the homotopy-invariant coincidence point theory. Let us begin, by respectively defining the quantities $MF(f)$ and $XF(f)$ as the minimum number and maximum number of fixed points among all maps homotopic to $f$.

For a self-map $f : X \longrightarrow X,$ we always have
\[0 \leq MF(f) \leq XF(f) \leq \#X.\]

Any one of the above inequalities can be strict or equality depending on the situation or conditions at hand.

\begin{definition}\cite{BoxerSta19fixed}
Let $f : X \longrightarrow X$ be a mapping on $X$. Then
\begin{enumerate} [label=(\roman*) ]
     \item The homotopy fixed point spectrum of $f$ is defined as:
\[S(f) = \{\#\textup{Fix($g$)} \, | \, g \simeq f \} \subseteq \{0, 1,  \ldots ,\#X\};\]
     \item The fixed point spectrum of $X$ is defined as:
\[F(X) = \{\# \mbox{Fix}(f) \, | \, f : X \longrightarrow X \mbox{ is continuous}\}.\]
\end{enumerate}
\end{definition}

For simplicity, in the sequel we shall be using ``continuous" instead of ``$\kappa$-continuous". Let's also denote a digital image $(X, \kappa)$ as simply $X,$ since we will not be referencing the adjacency relation explicitly, and we will often refer to ``digital images" as simply ``images". Moreover, from now on, we will consider the functions $f_1, f_2 : X \longrightarrow Y$ to be continuous maps between connected digital images $X$ and $Y$, unless stated otherwise.

Now, let's consider the set $C(f_1, f_2),$ which we call the ``coincidence point set" of the maps $f_1$ and $f_2.$
\[C(f_1, f_2) := \{x \in X \, | \, f_1(x) = f_2(x)\}.\]

Whenever we deform $f_1$ and $f_2$, the size and shape of $C(f_1, f_2)$ may vary greatly. However, in topological coincidence theory we are not interested in any such \emph{inessential} changes. We rather tend to capture only those features which remain unchanged by arbitrary homotopies.

In this paper, we are more concerned about the size of the set $C(f_1, f_2)$, and one possible tool to measure the set $C(f_1, f_2)$ is ``the minimum number of coincidence points" (i.e. $MC(f_1, f_2)$) which we define as: 
\[MC(f_1, f_2) := \min \{\# C(g_1, g_2) \, | \, g_1 \simeq f_1 \mbox{ and } g_2 \simeq f_2\}.\]



\begin{theorem} \label{th1}
Let $X,Y$ be isomorphic digital images and $f_1, f_2 : X \longrightarrow X$ be continuous mappings. Then there are continuous mappings $g_1,g_2 : Y \longrightarrow Y$ such that $\# C(f_1, f_2) \, \textup{=} \, \# C(g_1, g_2).$
\end{theorem}

\begin{proof}
Let $\Phi : X \longrightarrow Y$ be an isomorphism and $A = C(f_1, f_2).$ Since $\Phi$ is one-to-one, $\#\Phi(A) = \#A.$
Let $g_1, g_2 : Y \longrightarrow Y$ be defined by $g_i = \Phi \circ f_i \circ \Phi^{-1},$ for $i = 1, 2.$ Now, for an arbitrary $y_0 \in \Phi(A),$ let $x_0 = \Phi^{-1}(y_0).$ Then
\begin{equation}
  \begin{split}
      g_1(y_0) &= \Phi \circ f_1 \circ \Phi^{-1}(y_0)\\
               &= \Phi \circ f_1(x_0)\\
               &= \Phi \circ f_2(x_0)\\
               &= \Phi \circ f_2 \circ \Phi^{-1}(y_0)\\
               &= g_2(y_0).
  \end{split}
\end{equation}
Let $B = C(g_1, g_2)$, then it follows that
\[\Phi(A) \subseteq B,\]
hence \[\#A \leq \#B.\]
Similarly, let $y_1 \in B$ (arbitrary) and $x_1 = \Phi^{-1}(y_1).$ Then
\begin{equation}
  \begin{split}
      f_1(y_1) &= \Phi^{-1} \circ g_1 \circ \Phi(y_1)\\
               &= \Phi^{-1} \circ g_1(x_1)\\
               &= \Phi^{-1} \circ g_2(x_1)\\
               &= \Phi^{-1} \circ g_2 \circ \Phi(y_1)\\
               &= f_2(y_1).
  \end{split}
\end{equation}
It follows that \[\Phi^{-1}(B) \subseteq A.\]
Therefore \[\#B \leq \#A.\]
Thus \[\# C(f_1, f_2) \, \textup{=} \, \# C(g_1, g_2)\]
as required. \qed
\renewcommand{\qed}{}  
\end{proof}

In the next few paragraphs, we will recall some classical topological notions. Notably, the following theorem proves that any change in the coincidence set $C(f, g)$ that may be effected by deforming both $f$ and $g$ can also be effected by deforming just $f.$ However, this property might not necessarily be true in digital topological setting as we will discuss later.

\begin{theorem}\cite{Brooks72removing} \label{th3}
Let $f, g : X \longrightarrow Y$ be mappings of a topological space $X$ into a topological manifold $Y,$ and suppose that $f^{\prime}$ and $g^{\prime}$ are homotopic to $f$ and $g$ respectively. Then there is a map $f^{\prime \prime}$ homotopic to $f$, such that \[C(f^{\prime \prime}, g) = C(f^{\prime},g^{\prime}).\]
\end{theorem}

The following result is a consequence of Theorem \ref{th3}.
\begin{corollary}\cite{Brooks72removing}\label{cor2}
If we deform only one of the two maps $f, g$ in Theorem \ref{th3} by a homotopy while leaving the other fixed. Then \[MC^{*}(f, g) = MC(f,g).\]
\end{corollary}

For example. Let $f$ be a self map of $X.$ Then
\[MC(f, \textup{id}) = \min \{\#\textup{Fix($g$)} \, | \, g \simeq f\} = MF(f).\]
i.e. the minimum number of coincidence points coincide with the classical minimum number of fixed points which plays a central role in the classical topological fixed point theory.

The statement in Theorem \ref{th3} only holds for continuous maps on manifolds (or slightly more general spaces than that). However, it does not hold even for continuous maps on polyhedra. One of the major limitations of Nielsen coincidence theory is that there is no way of dealing with homotopy-invariant coincidence counting where only one map varies by homotopy. When the space is a manifold it's no more a problem because of Brooks result (Theorem \ref{th3}), but even when the space is a polyhedron there is really no way to proceed. 


So, we believe that it would be interesting to investigate whether or not Theorem \ref{th3} and Corollary \ref{cor2} holds in the setting of digital spaces. A partial answer to this problem is given in Proposition \ref{prop}.

Now, for some maps $f_1,f_2 : X \longrightarrow Y,$ we may define the following set $HCS(f_1, f_2),$ which we call the ``homotopy coincidence point spectrum" of the functions $f_1$ and $f_2$ as follows:
\[HCS(f_1, f_2) = \{\# C(g_1, g_2) \, | \, g_1 \simeq f_1 \mbox{ and } g_2 \simeq f_2\} \subseteq \{0, 1,  \ldots ,\#X\}.\]

\begin{remark}
\begin{enumerate}[label=(\roman*) ] 
    \item $MC(f_1, f_2) = \min{HCS(f_1, f_2)};$
    \item Moreover, both $MC(f_1, f_2)$ and $HCS(f_1, f_2)$ are homotopy invariants for any continuous functions $f_1$ and $f_2$.
\end{enumerate}
\end{remark}

Now, we may also consider the ``coincidence point spectrum" of $X$, which we define and denote as:
\[CS(X) = \{\# C(f_1, f_2) \, | \, f_1, f_2 : X \longrightarrow Y \mbox{ are continuous}\}.\]

The following immediately follows as a consequences of Theorem \ref{th1} above.
\begin{corollary}
Let $X$ and $Y$ be isomorphic digital images. Then \[CS(X) = CS(Y).\]
\end{corollary}

To avoid confusion, when we allow only one of the two maps to be deformed by a homotopy while keeping the other map fixed, we let $MC^{*}(f_1, f_2),$ $HCS^{*}(f_1, f_2)$ and $CS^{*}(X)$ to denote $MC(f_1, f_2), HCS(f_1, f_2)$ and $CS(X)$ respectively. For instance, we have
\[MC^{*}(f_1, f_2) := \min \{\# C(g_1, f_2) \, | \, g_1 \simeq f_1 \mbox{ and } f_2 \mbox{ is fixed}\}.\]

\begin{theorem} 
Suppose that $X$ is a rigid digital image. Let $id$ and $c$ be the identity and constant mappings respectively. Then \[HCS(id, c) = S(c) \mbox{ and } MC(id, c) = MF(c).\]
\end{theorem}

\begin{proof}
The results follows immediately from Corollary \ref{cor2} and the fact that $X$ is rigid.
\renewcommand{\qed}{}  
\end{proof}

\begin{example}
Let $X$ be a rigid digital image and $f : X \longrightarrow X$ be a continuous mapping. Then \[HCS(f, id) = S(f).\]
\end{example}

\begin{example} \label{exp1}
Let $X$ be a connected digital image, $f : X \longrightarrow X$ be a continuous mapping and the constant mapping $c$ be rigid. Then \[HCS(f, c) =  \{1\}.\]
\end{example}

\begin{remark}
In Example \ref{exp1} above, we realise that the assumption that $c$ is rigid is very strong and therefore forced $X$ to be a single point, which makes the example a little bit not too interesting.
\end{remark}

\begin{figure}    
\centering
\resizebox{2.5cm}{3cm}{%
\begin{tikzpicture}
			\tikzstyle{subj} = [circle, minimum width=8pt, fill, inner sep=0pt]
			\tikzstyle{subja} = [circle, minimum width=8pt, fill, inner sep=0pt]
			\tikzstyle{subjb} = [circle, minimum width=8pt, fill, inner sep=0pt]
			\tikzstyle{subjc} = [circle, minimum width=8pt, fill, inner sep=0pt]
			\tikzstyle{subje} = [circle, minimum width=8pt, fill, inner sep=0pt]
			\node[blue, subj, label=above:{\textcolor{blue}{1}}] (max) at (0,4.75) {};
			\node[blue, subja, label=above:{\textcolor{blue}{2}}] (a) at (-2,2.5) {};
			\node[blue, subjb, label=above:{\textcolor{blue}{5}}] (b) at (2,2.5) {};
			\node[blue, subjc, label=below:{\textcolor{blue}{3}}] (c) at (-1.5,0) {};
			\node[blue, subje, label=below:{\textcolor{blue}{4}}] (e) at (1.5,0) {};
			\draw (max) -- (a) -- (c) -- (e) -- (b) -- (max); 
\end{tikzpicture}}
\caption{The digital image $C_5$.}
\label{plot:speedup}
\end{figure}
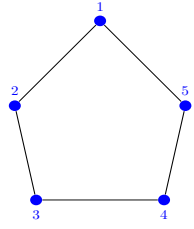

\begin{proposition} \label{prop}
If $X$ is a rigid image then Corollary \ref{cor2} holds.
\end{proposition}

\begin{example} \label{exp2}
Let $C_5$ be the cycle of $5$ points, $id$ be the identity map, and $c$ be a constant map. Then \[\#C(id, c) = 1.\]

If we change $id$ by homotopy to some other map say $f$, we will always have $\#C(f, c) = 1$ since $f$ must be a rotation \cite{BoxerSta19fixed}. Therefore, the spectrum of coincidences when we change only the first map by homotopy is just the set $\{1\}.$ i.e. $HCS^{*}(id, c) = \{1\}$.

However, if we are allowed to change both maps by homotopy, then we can change $c$ to some other map say $g$ which has $0, 1, 2$ or $3$ fixed points \cite{BoxerSta19fixed}. Now, the spectrum of coincidences when we are allowed to change both maps by homotopy is $\{0,1,2,3\}.$ i.e. $HCS(id, c) = \{0,1,2,3\}.$

Moreover, in this particular example, when we interchange the position of the two mappings, we have $HCS^{*}(c, id) = HCS(c, id) = \{0,1,2,3\}$.

In fact, any cycle of $5$ or more points can hold a similar result to this example. In other words, there's nothing special about the cycle of $5$ points specifically. We choose the cycle of $5$ points, $id$ and $c$ to emphasize that $HCS^{*}(f, g)$ can be different from $HCS(f, g)$ for any continuous maps $f$ and $g$.
\end{example}

\begin{remark}
\begin{enumerate}[label=(\roman*) ] 
    \item Example \ref{exp2} above shows that Theorem \ref{th3} is false in the setting of digital spaces;
    \item The equality $HCS^{*}(f_1, f_2) = HCS^{*}(f_2, f_1)$ need not necessarily always be true for any mappings $f_1$ and $f_2$.
\end{enumerate}
\end{remark}

\section{Common Fixed Point Spectrum}

In this section, we present the concept of common fixed point set, some related invariants and results were also discuss. Let $f_1,f_2 : X \longrightarrow X,$ we define the ``common fixed point set" of $f_1$ and $f_2$ as:
\[CF(f_1, f_2) := \{x \in X \, | \, f_1(x) = f_2(x) = x\}.\]

\begin{theorem} \label{thm5}
Let $X,Y$ be isomorphic digital images and $f_1, f_2 : X \longrightarrow X$ be continuous mappings. Then there are continuous mappings $g_1,g_2 : Y \longrightarrow Y$ such that $\# CF(f_1, f_2) \, \textup{=} \, \# CF(g_1, g_2).$
\end{theorem}

\begin{proof}
The result follows from similar argument to the proof of Theorem \ref{th1}. \qed
\renewcommand{\qed}{}  
\end{proof}

Similar to the assertions in the previous section, we define the ``minimum number of common fixed points" of $f_1$ and $f_2$ as:
\[MCF(f_1, f_2) := \min \{\# CF(g_1, g_2) \, | \, g_1 \simeq f_1 \mbox{ and } g_2 \simeq f_2\}.\]

Moreover, for some maps $f_1,f_2 : X \longrightarrow X,$ we may consider the following set $HFS(f_1, f_2),$ which we call the ``homotopy common fixed point spectrum" of $f_1$ and $f_2$:
\[HFS(f_1, f_2) = \{\# CF(g_1, g_2) \, | \, g_1 \simeq f_1 \mbox{ and } g_2 \simeq f_2\} \subseteq \{0, 1,  \ldots ,\#X\}.\]

We may also consider the following ``common fixed point spectrum" of $X$ defined as:
\[CFS(X) = \{\# CF(f_1, f_2) \, | \, f_1, f_2 : X \longrightarrow X \mbox{ are continuous}\}.\]

The following immediately follows as a consequences of Corrollary \ref{thm5} above.
\begin{corollary}
Let $X$ and $Y$ be isomorphic digital images. Then \[CFS(X) = CFS(Y).\]
\end{corollary}

\begin{remark} \label{rem1}
\begin{enumerate} [label=(\roman*) ]
     \item If $f_1 = f_2 = f$, then $CF(f_1, f_2) = \textup{Fix($f$)}$;
     \item It is easy to see that $F(X)$ is always a subset of $CFS(X)$.
\end{enumerate}
\end{remark}

By Remark \ref{rem1}, we obtain the following two corollaries.

\begin{corollary} \cite{BoxerSta19fixed}
Let $X,Y$ be isomorphic digital images and $f : X \longrightarrow X$ be continuous mapping. Then there exists a continuous mapping $g : Y \longrightarrow Y$ such that $\# \textup{Fix($f$)} \, \textup{=} \, \# \textup{Fix($g$)}.$
\end{corollary}

\begin{corollary} \cite{BoxerSta19fixed}
Let $X$ and $Y$ be isomorphic digital images. Then \[F(X) = F(Y).\]
\end{corollary}

\begin{question}\label{ques}
If $f_1 = f_2 = f$, do we always have
\begin{enumerate} [label=(\roman*) ]
     \item $MCF(f_1, f_2) = MF(f)$?
     \item $HFS(f_1, f_2) = S(f)$?
     \item $CFS(X) = F(X)$?
\end{enumerate}
\end{question}

In response to Questions \ref{ques}, we consider $X \subset \mathbb{Z}^3$ to be a digital image of unit cube of 8 points with 6-adjacency as shown in Figure \ref{fig3}. For any continuous mapping $f: X \longrightarrow X$, $S(f) = \{0,1,2,3,4,5,6,8\} = F(X)$ since $X$ is contractible \cite{BoxerSta19fixed}. Further, since $HFS(c, c) = \{0,1,2,3,4,5,6,8\}$ and $CFS(X) = \{0,1,2,3,4,5,6,8\},$ we have $HFS(c, c) = S(c) = F(X) = CFS(X).$  This further implies that $HFS(f, f) = CFS(X)$ for any continuous mapping $f: X \longrightarrow X.$

\begin{figure} \label{fig3}
\centering
\resizebox{3.5cm}{3cm}{%
\begin{tikzpicture}
	\draw[step=1cm,white,very thin] (0,0) grid (3,3);
    \draw [blue, thick] (0,0) -- (0,2);
    \draw [blue, thick] (0,2) -- (1,3);
    \draw [blue, thick] (1,3) -- (3,3);
    \draw [blue, thick] (1,1) -- (1,3);
    \draw [blue, thick] (0,0) -- (2,0);
    \draw [blue, thick] (0,2) -- (2,2);
    \draw [blue, thick] (3,1) -- (3,3);
    \draw [blue, thick] (2,0) -- (2,2);
    \draw [blue, thick] (2,0) -- (3,1);
    \draw [blue, thick] (1,1) -- (3,1);
    \draw [blue, thick] (2,2) -- (3,3);
    \draw [blue, thick] (0,0) -- (1,1);
    \filldraw[blue] (0,0) circle (2pt) node[left]{$x_0$};
    \filldraw[blue] (1,1) circle (2pt) node[left]{$x_3$};
    \filldraw[blue] (0,2) circle (2pt) node[left]{$x_1$};
    \filldraw[blue] (2,2) circle (2pt) node[right]{$x_5$};
    \filldraw[blue] (2,0) circle (2pt) node[below]{$x_4$};
    \filldraw[blue] (1,3) circle (2pt) node[above]{$x_2$};
    \filldraw[blue] (3,3) circle (2pt) node[right]{$x_6$};
    \filldraw[blue] (3,1) circle (2pt) node[right]{$x_7$};
\end{tikzpicture}}
\caption{A contractible 3-dimensional digital image with a 6-adjacency relation.}
\end{figure}
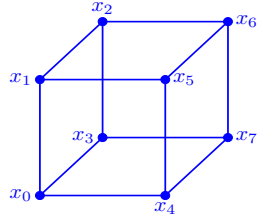

\begin{conjecture}
Let $X$ be a contractible image and $f : X \longrightarrow X$ be a continuous mapping. Then, $HFS(f, f) = S(f)$ and  $CFS(X) = F(X)$.
\end{conjecture}

\section{Retracts of $X$}

In this section, we study how retractions interact with the coincidence and common fixed point spectra. To begin with, it is natural to ask whether or not, whenever $A$ is a subset of an image $X$, we will have $CS(A) \subseteq CS(X)$. The answer is negative as shown by the following example.

\begin{example}
Let $X$ be the digital image in Fig. \ref{fig3}. If $A = X \backslash \{x_0\}$ then $A \subset X$ and $CS(A)=\{0,1,2,3,4,5,6,7\} \not\subseteq \{0,1,2,3,4,5,6,8\} =CS(X).$
\end{example}

However, if $A \subset X$ is a retract of $X$, then we will have an affirmative answer as shown by Theorem \ref{th4} below.

\begin{definition} \cite{Boxer94digitally}
Let $A$ be a subset of a digital image $X.$ A continuous function $r : X \longrightarrow A$ is called a retraction, and $A$ is a retract of $X,$ if $r(a) = a$ for all $a \in A.$ Moreover, $r$ is called a $\kappa$-deformation retraction, and $A$ is a $\kappa$-deformation retract of $X,$ if $r$ satisfies $i \circ r \simeq_{\kappa} id,$ where $i : A \longrightarrow X$ is the inclusion map. 
\end{definition}

\begin{theorem} \label{th4}
Let $A$ be a retract of an image $X.$ Then \[CS(A) \subseteq CS(X).\]
\end{theorem}

\begin{proof}
Let $f_1, f_2 : A \longrightarrow A$ be continuous functions and $r : X \longrightarrow A$ be a retraction mapping. Now, we define the functions $g_1, g_2 : X \longrightarrow X$ as $g_1 = i \circ f_1 \circ r$  and $g_2 = i \circ f_2 \circ r,$  where $i : A \longrightarrow X$ is the inclusion map. So, from Theorem \ref{th2}, the functions $g_1$ and $g_2$ are continuous. Therefore, we have $g_1(x) = f_1(x)$ if and only if $x \in A$ and similarly $g_2(x) = f_2(x)$ if and only if $x \in A.$ Thus $C(f_1, f_2) \, \textup{=} \, C(g_1, g_2),$ hence the assertion follows immediately since $f_1$ and $f_2$ are arbitrarily chosen. \qed
\renewcommand{\qed}{}  
\end{proof}

\begin{theorem} \label{cor}
Let $A$ be a retract of an image $X.$ Then \[CFS(A) \subseteq CFS(X).\]
\end{theorem}

\begin{proof}
The assertions follows from a similar argument to the proof of Theorem \ref{th4} .
\end{proof}


\begin{corollary} \cite{BoxerSta19fixed}
Let $A$ be a retract of an image $X.$ Then \[F(A) \subseteq F(X).\]
\end{corollary}

\section{Divergence Degree}

In this section, we introduce the notion of divergence degree of a point $x$ in an image $X$, which give us an estimate of ``non-coincident indicator of the point $x$". 

Throughout this section, $f_1$ and $f_2$ are self maps on $X$. We begin with presenting an important definition we use to define the degree at which two given functions differ at a point $x$. This we call the complement of the coincidence point set, which we denote by $\overline{C(f_1, f_2)}$ and define as:
\[\overline{C(f_1, f_2)} := \{x \in X \, | \, f_1(x) \not= f_2(x)\}.\]
Whenever $f_1(x) \not= f_2(x)$, we say that $f_1$ and $f_2$ does not meet at point $x$ in $X.$

\begin{definition} 
Let $(X, \kappa)$ be a digital image with $\# X > 1$ and $x \in X.$ Then the ``non-coincident indicator of $x$" which we call the ``Divergence Degree of $x$" is define as:
\[D(x) := \min \{\#\overline{C(f_1, f_2)} \, | \, f_1(x) \not= f_2(x) \mbox{ and } f_1,f_2 \mbox{ are continuous} \}.\]
\end{definition}

\begin{theorem} 
Let $X$ be a connected digital image with $n = \#X > 1.$ Then \[n-1 \in CS(X) \mbox{ if and only if there is some } x \in X \mbox{ with } D(x) = 1.\]
\end{theorem}

\begin{proof}
It is not too difficult to see that $n-1 \in CS(X)$ if and only if there exist $f_1, f_2 \in C(X, \kappa)$ with exactly $n-1$ coincidence points. i.e. the only $x \in X$ not coincident by $f_1$ and $f_2$ has $D(x) = 1.$ Hence proving the result.
\qed
\renewcommand{\qed}{}  
\end{proof}

\begin{example}
Let $X$ be the digital image in Fig. \ref{fig1}. Let $f,p_v$ and $p_h$ be self maps on $X$ representing; $180^{\circ}$ rotation of $X$, vertical flip of $X$ and horizontal flip of $X$ respectively, then $f,p_v$ and $p_h$ are continuous. Let $g_1,g_2,g_3 : X \longrightarrow X$ be mappings define as:
\[g_1(7) = 5, g_1(11) = 10, g_1(18) = 16, g_1(x) = x \mbox{ for } x \in X \backslash \{7, 11, 18\},\]
\[g_2(1) = 3, g_2(8) = 9, g_2(12) = 14, g_2(x) = x \mbox{ for } x \in X \backslash \{1, 8, 12\}\] and
\[g_3(1) = 3, g_3(7) = 5, g_3(8) = 9, g_3(11) = 10, g_3(12) = 14, g_3(18) = 16, g_3(x) = x\] 
for $x \in X \backslash \{1, 7, 8, 11, 12, 18\}.$ Then $g_1,g_2$ and $g_3$ are all continuous.

Let $h_1 : X \longrightarrow X$ be the mappings that maps the top bar into the bottom bar and fixes all other points and $h_2 : X \longrightarrow X$ be the mappings that maps the bottom bar into the top bar and fixes all other points. These are all the possible non trivial (different from $id$ and $c$) continuous functions on $X$ providing different coincidence point sets.

So, after some computations we obtain  
\[D(x) = 3 \hspace{0.5cm} \mbox{ for } x \in \{1,7,8,11,12,18\},\] \[D(x) = 14 \hspace{0.5cm} \mbox{ for } x \in \{2,3,5,6,13,14,16,17\}\] and \[D(x) = 17 \hspace{0.5cm} \mbox{ for } x \in \{4,9,10,15\}.\]
\end{example}

\section{Conclusion}
In this article, we introduced, studied and investigated some properties of the coincidence point set of digitally continuous maps. Following the Rosenfeld graphical model which seems more combinatorial than topological, we achieved results that are not analogous to the classical topological fixed point theory, for instance, in classical coincidence theory the only interesting homotopy invariant count of the number of coincidence points is $MC(f_1,f_2)$. Whereas, here we introduced $HCS(f_1,f_2)$ which is not studied in the classical coincidence theory. We also introduced and studied some topological invariants related to coincidence and common fixed point sets for continuous maps on a digital image. Moreover, we studied how these coincidence point sets are affected by rigidity and deformation retraction. Also, we briefly introduced the concept of divergence degree of a point in a digital image and illustrated by example that D(x) can assume different values for different choice of point. Lastly, we are optimistic that these properties will be applicable in image processing and its related disciplines in the nearest future.

\section{Acknowledgement}
The authors acknowledge the financial support provided by the Center of Excellence in Theoretical and Computational Science (TaCS-CoE), Faculty of Science, KMUTT. The first and the third authors were supported by ``the Petchra Pra Jom Klao Ph.D. Research Scholarship" from `King Mongkut's University of Technology Thonburi" (Grant No. 35/2017 and 38/2018 respectively). Finally, the authors would like to thank Assoc. Prof. Peter Christoper Staecker for his careful reading and his valuable suggestions to the improvement of this paper, especially his idea of Example \ref{exp2}.

%
%
%
%

\bibliographystyle{plain} 
\bibliography{reference}         

\end{document}